\title{Random Graphons and a Weak Positivstellensatz for Graphs}
\author{{\sc L\'aszl\'o Lov\'asz\footnote{Research sponsored by OTKA Grant No.~67867.}}~ and {\sc Bal\'azs Szegedy}\\
Institute of Mathematics, E\"otv\"os Lor\'and University\\
Budapest, Hungary, and\\
Department of Mathematics, University of Toronto\\
Toronto, Ontario, Canada}
\date{February 2009}

\documentclass{article}
\usepackage{amssymb,amsmath,hyperref}
\usepackage{theorem}

\nonstopmode

\long\def\killtext#1{}
\newtheorem{theorem}{Theorem}[section]
\newtheorem{prop}[theorem]{Proposition}

\theorembodyfont{\rmfamily}

\newenvironment{proof}{\medskip\noindent{\bf Proof. }}{\hfill$\square$\medskip}

\begin{document}

\addtolength{\baselineskip}{3pt} \setlength{\oddsidemargin}{0.2in}

\def\R{{\mathbb R}}
\def\one{{\sf\bf 1}}
\def\Q{{\mathbb Q}}
\def\Z{{\mathbb Z}}
\def\N{{\mathbb N}}
\def\C{{\mathbb C}}
\def\U{{\mathbb U}}
\def\hom{{\rm hom}}
\def\iso{{\rm iso}}
\def\inj{{\rm inj}}
\def\surj{{\rm sur}}
\def\Inj{{\rm Inj}}
\def\neighb{{\rm neighb}}
\def\inj{{\rm inj}}
\def\ind{{\rm ind}}
\def\sur{{\rm sur}}
\def\PAG{{\rm PAG}}
\def\eps{\varepsilon}
\def\Ge{{\mathbf G}}
\def\Ha{{\mathbf H}}
\def\Aa{{\mathbf A}}
\def\CUT{\text{\rm CUT}}
\def\IO{{\infty\to1}}
\def\GR{\text{\rm GR}}
\def\CLQ{\text{\rm CLQ}}
\def\LT{{\text{\rm LEFT}}}
\def\RT{{\text{\rm RIGHT}}}
\def\Haus{{\rm Hf}}

\def\Fi{\mathbf{\Phi}}
\def\mul{{\rm mult}}
\def\flat{{\rm flat}}
\def\maxcut{{\sf maxcut}}
\def\MaxCut{{\sf MaxCut}}
\def\id{{\rm id}}
\def\irreg{{\rm irreg}}

\def\wh{\widehat}
\def\tv{{\rm tv}}
\def\tr{{\rm tr}}
\def\cost{\hbox{\rm cost}}
\def\val{\hbox{\rm val}}
\def\rk{{\rm rk}}
\def\diam{{\rm diam}}
\def\Ker{{\rm Ker}}
\def\QG{{\cal QG}}
\def\QGM{{\cal QGM}}
\def\CD{{\cal CD}}
\def\Pr{{\sf P}}
\def\E{{\sf E}}
\def\Var{{\sf Var}}
\def\Ent{{\sf Ent}}
\def\T{{^\top}}
\def\PD{{\sf Pd}}

\def\proofend{\hfill$\square$\medskip}

\def\AA{{\cal A}}\def\BB{{\cal B}}\def\CC{{\cal C}}
\def\DD{{\cal D}}\def\EE{{\cal E}}\def\FF{{\cal F}}
\def\GG{{\cal G}}\def\HH{{\cal H}}\def\II{{\cal I}}
\def\JJ{{\cal J}}\def\KK{{\cal K}}\def\LL{{\cal L}}
\def\MM{{\cal M}}\def\NN{{\cal N}}\def\OO{{\cal O}}
\def\PP{{\cal P}}\def\QQ{{\cal Q}}\def\RR{{\cal R}}
\def\SS{{\cal S}}\def\TT{{\cal T}}\def\UU{{\cal U}}
\def\VV{{\cal V}}\def\WW{{\cal W}}\def\XX{{\cal X}}
\def\YY{{\cal Y}}\def\ZZ{{\cal Z}}

\maketitle

\begin{abstract}
In an earlier paper the authors proved that limits of convergent
graph sequences can be described by various structures, including
certain 2-variable real functions called graphons, random graph
models satisfying certain consistency conditions, and normalized,
multiplicative and reflection positive graph parameters. In this
paper we show that each of these structures has a related, relaxed
version, which are also equivalent. Using this, we describe a further
structure equivalent to graph limits, namely probability measures on
countable graphs that are ergodic with respect to the group of
permutations of the nodes.

As an application, we prove an analogue of the Positivstellensatz for
graphs: We show that every linear inequality between subgraph
densities that holds asymptotically for all graphs has a formal proof
in the following sense: it can be approximated arbitrarily well by
another valid inequality that is a ``sum of squares'' in the algebra
of partially labeled graphs.
\end{abstract}

\tableofcontents

\section{Introduction}

In an earlier paper the authors proved that limits of convergent
graph sequences can be described by various structures, including
2-variable symmetric, measurable functions $[0,1]^2\to[0,1]$, random
graph models satisfying a ``consistency'' and a ``locality''
condition, and normalized, multiplicative and reflection positive
graph parameters (see Theorem \ref{THM:LIM-CHAR} and Proposition
\ref{PROP:DAG-SEMI-1}).

In this paper we show that each of these structures has a related,
relaxed version: We can drop the multiplicativity condition on the
graph parameter, replacing it with the simple condition that deleting
isolated nodes does not change the value of the parameter. We can
drop the ``locality'' condition on the random graph model. We can
replace the graphon by a probability distribution of the graphon. As
the first main result of this paper, we prove that these relaxed
versions are also equivalent.

This result will be used in adding a further equivalent structure to
the list of structures describing graph limits: a probability measure
on countable graphs that is ergodic with respect to the group of
permutations of the nodes.

As an application, we prove an analogue of the Positivstellensatz for
graphs. Many fundamental theorems in extremal graph theory can be
expressed as linear inequalities between subgraph densities. For
example, the Mantel--Tur\'an Theorem is implied by the linear
inequality that the density of triangles is always at least the
edge-density minus $\frac12$. (To be more precise, using
``homomorphism densities'' to be defined in Section \ref{PRELIM}, we
get inequalities that hold true for all graphs; in terms of subgraph
densities, we get in general only asymptotic results with some error
terms.)

It has been observed long ago that most of these extremal results
seem to follow by one of more tricky applications of the
Cauchy--Schwartz inequality. We confirm this in the following sense:
we show that every linear inequality between homomorphism densities
that holds for all graphs can be derived, up to an arbitrarily small
error term, by the Cauchy--Schwartz Inequality. To make the last
phrase precise, we use graph algebras introduced by Freedman,
Lov\'asz and Schrijver in \cite{FLS}. The square of an algebra
element, when expanded, yields a valid linear inequality between
homomorphism densities. Sums of such inequalities yield further valid
linear inequalities, and our result says that such sums of squares
are dense among all valid linear inequalities.

\section{Preliminaries}\label{PRELIM}

\subsection{Homomorphism densities and limits}

In this paper, all graphs are simple. If we don't quantify, we also
mean that the graph is finite.

For two graphs $F$ and $G$, we write $F\cong G$ if they are
isomorphic, and $F\simeq G$ if they become isomorphic after their
isolated nodes are deleted. So the graph $U_n$ consisting of $n$
isolated nodes satisfies $U_n\simeq K_0\cong U_0$.

For two graphs $F$ and $G$, let $\hom(F,G)$ denote the number of
homomorphisms (adjacency-preserving maps) from $F$ to $G$, and
$\inj(F,G)$, the number of injective homomorphisms from $F$ to $G$.
We consider the {\it homomorphism densities}
\[
t(F,G)=\frac{\hom(F,G)}{|V(G)|^{|V(F)|}},
\]
and {\it subgraph densities}
\[
t_\inj(F,G)=\frac{\inj(F,G)}{|V(G)|\cdot(|V(G)|-1)\cdots(|V(G)|-|V(F)|+1)},
\]

Let $\WW_0$ denote the set of symmetric measurable functions
$W:~[0,1]^2\to[0,1]$. A {\it graphon} is any function in $\WW_0$. For
every graph $F$ and graphon $W$, we define the density of $F$ in $W$
by
\[
t(F,W)=\int_{[0,1]^V} \prod_{ij\in E} W(x_i,x_j)\,\prod_{i\in V} dx_i
\]

To every graph $G$ we can assign a graphon $W_G$ as follows: Let
$V(G)=[n]$. Split $[0,1]$ into $n$ intervals $J_1,\dots,J_n$ of
length $\lambda(J_i)=\alpha_i/\alpha_G$. For $x\in J_i$ and $y\in
J_j$, let $W_G(x,y)=\one_{ij\in E(G)}$. With this construction, we
have $t(F,G)=t(F,W_G)$ for all finite graphs $F$.

We consider on $\WW_0$ the {\it cut norm}
\[
\|W\|_\square = \sup_{S,T\subseteq[0,1]}\Bigl|\int_{S\times T}
W(x,y)\,dx\,dy\Bigr|
\]
where the supremum is taken over all measurable subsets $S$ and $T$,
and the {\it cut distance}
\[
\delta_\square(U,W)=\inf_{\phi,\psi} \|U^\phi-W^\psi\|_\square,
\]
where $\phi,\psi$ range over all measure preserving maps from
$[0,1]\to[0,1]$, and $W^\phi(x,y)=W(\phi(x),\phi(y))$
\cite{BCLSV0,BCLSSV}. This also defines a distance between graphs by
\[
\delta_\square(F,G)=\delta_\square(W_F,W_G).
\]
(See \cite{BCLSV1} for more combinatorial definitions of this graph
distance.)

We note that $\delta_\square(U,W)=0$ can hold for two different
graphons: $\delta_\square(W^\phi,W^\psi)=0$ for every graphon $W$ and
measure preserving maps $\phi, psi:~[0,1]\to[0,1]$. (It was proved in
\cite{BCL} that this gives all pairs of graphons with distance $0$.)
We call two graphons {\it weakly isomorphic} if their distance is
$0$.

It was proved in \cite{LSz3} that $(\WW_0,\delta_\square)$ is a
compact metric space.

A sequence of graphs $(G_n)$ with $|V(G_n)|\to\infty$ is {\it
convergent} if the densities $t(F,G_n)$ converge for all finite
graphs $F$. This is clearly equivalent to saying that the subgraph
densities $t_\inj(F,G_n)$ converge for all finite graphs $F$.

It was proved in \cite{BCLSV1} that a graph sequence is convergent if
and only if it is Cauchy in the $\delta_\square$ distance. It was
proved in \cite{LSz1} that for every convergent graph sequence there
is a limit object in the form of a function $W\in\WW$, so that
\[
t(F,G_n)\to t(F,W)\qquad \text{for all graphs}~F.
\]
In \cite{BCLSV1} it was shown that this is equivalent to
$\delta_\square(W_{G_n},W)\to 0$. In \cite{BCL} it was proved that
this limit is uniquely determined up to weak isomorphism.

\subsection{Partially labeled graphs and quantum graphs}

A {\it $k$-labeled graph} is a graph in which $k$ of the nodes are
labeled by $1,\dots,k$ (there may be any number of unlabeled nodes).
A $0$-labeled graph is just an unlabeled graph. Let $\FF_k$ denote
the set of $k$-labeled graphs (up to label-preserving isomorphism).

A $k$-labeled graph $F$ is called {\it flat} if $V(F)=[k]$. Let
$\FF'_k$ denote the set of all flat $k$-labeled graphs.

Let $F_1$ and $F_2$ be two $k$-labeled graphs. We define the
$k$-labeled graph $F_1F_2$ by taking their disjoint union, and then
identifying nodes with the same label (if multiple edges arise, we
only keep one copy). Clearly this multiplication is associative and
commutative. For two $0$-labeled graphs, $F_1F_2$ is their disjoint
union.

Sometimes it is more convenient to combine $k$-labeled graphs into a
single structure. A {\it partially labeled graph} is a finite graph
in which some of the nodes are labeled by distinct positive integers.
For two partially labeled graphs $F_1$ and $F_2$, let $F_1F_2$ denote
the partially labeled graph obtained by taking their disjoint union,
and identifying nodes with the same label. Let $\FF^*$ denote the set
of partially labeled graphs (up to isomorphism).

A {\it quantum graph} is defined as a formal linear combination of
graphs with real coefficients. A {\it $k$-labeled quantum graph} is
defined similarly as a formal linear combination of $k$-labeled
graphs. The product of $k$-labeled graphs defined above extends to
quantum graphs by distributivity: if $f=\sum_{i=1}^n \lambda_i F_i$
and $g=\sum_{j=1}^m \mu_j G_j$, then $fg=\sum_{i=1}^n\sum_{j=1}^m
\lambda_i\mu_j  F_iG_j$.

\subsection{Graph parameters}

A {\it graph parameter} is a real valued function defined on
isomorphism types of graphs (including the graph $K_0$ with no nodes
and edges). Let $f$ be any graph parameter and fix an integer $k\ge
0$. We define the $k$-th {\it connection matrix} of the graph
parameter $f$ as the (infinite) symmetric matrix $M(f,k)$, whose rows
and columns are indexed by (isomorphism types of) $k$-labeled graphs,
and the entry in the intersection of the row corresponding to $F_1$
and the column corresponding to $F_2$ is $f(F_1F_2)$. The {\it flat
connection matrix} $M_\flat(f,k)$ is the submatrix of $M(f,k)$ formed
by rows and columns corresponding to flat $k$-labeled graphs (this
matrix is finite).

We denote by $\MM$ the space of $\FF^*\times\FF^*$ matrices (these
are infinite matrices). For a graph parameter $f$, we define the {\it
full connection matrix} as the symmetric matrix $M(f)\in\MM$, whose
entry in the intersection of the row corresponding to $F_1$ and the
column corresponding to $F_2$ is $f(F_1F_2)$. Clearly this matrix
contains as a submatrix all connection matrices $M(f,k)$. In the
other direction, we note that every finite submatrix of $M(f)$ is
contained as a submatrix in one of the matrices $M(f,k)$.

Let $f$ be a graph parameter. We say that $f$ is {\it
isolate-indifferent} if $f(G)=f(G')$ whenever $G\simeq G'$. The
parameter is {\it multiplicative} if $f(FG)=f(F)f(G)$, where $FG$
denotes the disjoint union of the graphs $F$ and $G$.

For every graph parameter $f$, we define its {\it M\"obius transform}
$f^\dag$ by
\[
f^\dag(F)=\sum_{F':\,V(F')=V(F)\atop E(F')\supseteq E(F)}
(-1)^{|E(F')\setminus E(F)|}f(F').
\]

We say that $f$ is {\it normalized} if $f(K_0)=f(K_1)=1$. Note that
for a multiplicative parameter, it would be enough to assume
$f(K_1)=1$, while for an isolate-indifferent parameter, it would be
enough to assume $f(K_0)=1$. Trivially, if a graph parameter is
multiplicative and normalized, then it is isolate-indifferent.

We call a graph parameter {\it reflection positive} if all of its
connection matrices are positive semidefinite (this is equivalent to
saying that its full connection matrix $M(f)$ is positive
semidefinite). We call it {\it flatly reflection positive} if all its
flat connection matrices are positive semidefinite.

We denote by $\KK$ the linear space of matrices $A\in\MM$ in which
$A_{F_1,G_1}=A_{F_2,G_2}$ if $F_1G_1\cong F_2G_2$, and by $\LL$, the
linear space of matrices $A\in\MM$ in which $A_{F_1,G_1}=A_{F_2,G_2}$
if $F_1G_1\simeq F_2G_2$. Clearly connection matrices define a
bijection between matrices in $\KK$ and graph parameters. Under this
bijection, matrices in $\LL$ correspond to isolate-indifferent graph
parameters.

Let $\PP\subseteq\MM$ denote the cone of positive semidefinite
matrices in $\MM$. Reflection positive graph parameters correspond to
matrices in $\PP\cap\KK$.

\subsection{Random graph models}

A {\it random graph model} is a sequence $(\Pr_n:~n=0,1,2,\dots)$,
where $\Pr_n$ is a probability distribution on graphs on $[n]$. Let
$\Ge_n$ be a random graph from distribution $\Pr_n$. We say that the
random graph model is {\it consistent}, if the distribution $\Pr_n$
is invariant under relabeling nodes, and if we delete node $n$ from
$\Ge_n$, the distribution of the resulting graph is the same as the
distribution of $\Ge_{n-1}$.

We say that the random graph model is {\it local}, if for every
$S\subseteq[n]$, the subgraphs of $\Ge_n$ induced by $S$ and
$[n]\setminus S$ are independent (as random variables).

Let $\binom{\N}{2}$ denote the set of all unordered pairs from $\N$.
Every subset of $\binom{\N}{2}$ can be thought of as a graph on node
set $\N$, and $\{0,1\}^{\binom{\N}{2}}$ is the set of all graphs on
$\N$. Let $\AA$ denote the $\sigma$-algebra on
$\{0,1\}^{\binom{\N}{2}}$ generated by the sets obtained by fixing
whether a given pair is connected or not.

A {\it random countable graph model} is a probability distribution
$\Pr$ on $(\{0,1\}^{\binom{\N}{2}},\AA)$. Such a distribution is {\it
consistent} if the distribution of the labeled subgraph induced by an
ordered finite set $S$ depends only on the size of $S$. The
distribution is {\it local} if for any two finite disjoint subsets
$S_1,S_2\subseteq\N$, the subgraphs induced by $S_1$ and $S_2$ are
independent (as random variables). The distribution is {\it
invariant} if it is invariant under permutations of $\N$. The
distribution is {\it ergodic} if there is no set $S\in\AA$ with
$0<\pi(S)<1$ invariant under permutations of $\N$. Invariant measures
form a convex set in the linear space of all signed measures, and
ergodic measures are the extreme points of this convex set.

A probability distribution on the Borel sets of
$(\WW_0,\delta_\square)$ will be called a {\it random graphon model}.
Note that the $\sigma$-algebra of Borel sets does not distinguish
weakly isomorphic graphons.

\section{Equivalent forms of the limit object}\label{LIMITFORMS}

\subsection{Graph limits and random graph limits}\label{RANDLIM}

We quote the following theorem, which was proved essentially in
\cite{LSz1}.

\begin{theorem}\label{THM:LIM-CHAR}
The following are equivalent (cryptomorphic):

\smallskip

{\rm(a)} A multiplicative, normalized graph parameter with
nonnegative M\"obius transform;

\smallskip

{\rm(b)} A consistent and local random graph model;

\smallskip

{\rm(c)} A consistent and local random countable graph model;

\smallskip

{\rm(d)} A graphon, up to weak isomorphism.

\smallskip

{\rm(e)} A point in the completion of the set of finite graphs with
the cut-metric;
\end{theorem}

The following theorem shows that in each of these objects, we can
naturally relax the conditions, to get another important set of
cryptomorphic structures.

\begin{theorem}\label{THM:LIM-CHAR1}
The following are equivalent (cryptomorphic):

\smallskip

{\rm(a)} An isolate-indifferent, normalized graph parameter with
nonnegative M\"obius transform;

\smallskip

{\rm(b)} A consistent random graph model;

\smallskip

{\rm(c)} A consistent random countable graph model;

\smallskip

{\rm(d)} A random graphon model.
\end{theorem}

\begin{proof}
We describe a cycle of constructions, mapping one object in the
theorem to the next.

(a)$\to$(b). Let $f$ be an isolate-indifferent, reflection positive,
normalized graph parameter with nonnegative M\"obius transform. Using
that $f$ is isolate-indifferent, we get
\[
\sum_{F:\,V(F)=[n]} f^{\dag}(F) = f(U_n) = 1.
\]
So we can construct a random graph $\Ge_n$ on $[n]$ by
\begin{equation}\label{EQ:FDAGRAND}
\Pr(\Ge_n=F)=f^\dag(F) \qquad (V(F)=[n]).
\end{equation}
It is clear that this distribution does not depend on the labeling of
the nodes. Let $F_0$ be a graph on $[n-1]$, and let $F_0^+$ be
obtained from $F_0$ by adding $n$ as an isolated node. Then
\begin{align*}
\Pr(\Ge_n\setminus \{n\}=F_0) & = \sum_{F:~F\setminus
\{n\}=F_0}\Pr(\Ge_n=F) = \sum_{F:~F\setminus \{n\}=F_0} f^\dag(F)\\
&= \sum_{F:~F\setminus \{n\}=F_0}\sum_{F'\supseteq F}
(-1)^{|E(F')|-|E(F)|}f(F')\\
&= \sum_{F'\supseteq F_0^+}f(F') \sum_{F\subseteq F'\atop F\setminus
\{n\}=F_0} (-1)^{|E(F')|-|E(F)|}
\end{align*}
Here the last sum is $0$ unless $F'$ contains no edges incident with
the node $n$, and so $f(F')=f(F'')$, where $F''=F'\setminus \{n\}$.
Thus
\[
\Pr(\Ge_n\setminus \{n\}=F_0) = \sum_{F''\supseteq F_0}f(F'')
(-1)^{|E(F'')|-|E(F_0)|}=f^\dag(F_0).
\]
Thus this model is consistent. We note that $f$ can be recovered by
\begin{equation}\label{EQ:FRAND}
f(F)=\Pr(F\subseteq \Ge_n)\qquad (V(F)=[n]).
\end{equation}

(b)$\to$(c). Let $\Ge_n$ be a random graph from a consistent finite
random graph model, we construct a countable random graph model by
$\pi(A_F)=\Pr(\Ge_n=F)$ ($V(F)=[n]$). This extends to a probability
measure on the $\sigma$-algebra $\AA$. It is straightforward to check
that this measure is consistent.

\smallskip

(c)$\to$(d). Let $\Ge$ be a random countable graph from a consistent
countable random graph model, we construct a probability distribution
on the Borel sets of $(\WW_0,\delta_\square)$. Let $\Ge_n$ be the
finite graph spanned by the first $n$ nodes of $\Ge$.

We claim that with probability $1$, the graph sequence $(\Ge_n)$ is
convergent. Theorem 2.11 in \cite{BCLSV1} implies that
\[
\delta_\square(\Ge_n,\Ge_m)\le \frac{10}{\sqrt{\log n}}
\]
with probability $1-\exp(-n^2/(2\log n))$.

Let $\Ha_k=\Ge_{2^k}$, then
\[
\Pr\Bigl(\delta_\square(\Ha_k,\Ha_{k+1})> \frac{10}{2^{k/2}}\Bigr) <
\exp\Bigl(\frac{-2^{2^k}}{2^{k+1}}\Bigr),
\]
and so by the Borel-Cantelli Lemma,
\[
\delta_\square(\Ha_k,\Ha_{k+1})\le \frac{10}{2^{k/2}}
\]
holds for all but a finite number of values of $k$, with probability
$1$. Hence with probability $1$, the sequence $(W_{\Ha_k})$ is a
Cauchy sequence in $(\WW_0,\delta_\square)$.

Now for a general value of $n$, let $k_n=\lceil \log\log n\rceil$.
Then as before, we get that
\[
\Pr\Bigl(\delta_\square(\Ge_n,\Ha_{k_n})>\frac{10}{\sqrt{\log
n}}\Bigr) < \exp\Bigl(\frac{-n^2}{2\log n}\Bigr).
\]
Again by the Borel-Cantelli Lemma,
\[
\delta_\square(\Ge_n,\Ha_{k_n})\le \frac{10}{\sqrt{\log n}}
\]
holds for all but a finite number of $n$, with probability $1$. This
proves that the sequence $(\Ge_n)$ is Cauchy. Thus it tends to a
limit graphon $\mathbf{W}$.

So we have described a method to generate a random graphon
$\mathbf{W}$. For every graph $F$, this satisfies
\[
t(F,\mathbf{W})=\lim_{n\to\infty} t(F, \Ge_n)=\lim_{n\to\infty}
t_\inj(F, \Ge_n).
\]
By the consistency of $\Ge$, the expectation of $t_\inj(F, \Ge_n)$ is
independent of $n$ for $n\ge k=|V(F)|$, and so
\[
\E(t(F,\mathbf{W}))=\lim_{n\to\infty} \E(t_\inj(F,
\Ge_n))=\E(t_\inj(F, \Ge_k))= \Pr(F\subseteq \Ge_k).
\]

(d)$\to$(a). Let $\mathbf{W}$ be a random graphon from any
probability distribution on the Borel sets of
$(\WW_0,\delta_\square)$. This defines a graph parameter $f$ by
\[
f(F)=\E(t(F,\mathbf{W}).
\]
For every fix $W\in\WW_0$, the graph parameter $f(.)=t(.,W)$ is
normalized, isolate-indifferent (since it is multiplicative), and has
nonnegative M\"obius transform (by Theorem \ref{THM:LIM-CHAR}).
Trivially, these properties are inherited by the expectation.
\end{proof}

\subsection{More equivalences}

In theorems \ref{THM:LIM-CHAR} and \ref{THM:LIM-CHAR1}, we listed
several seemingly quite different objects that have turned out
equivalent. In this section we show that these objects have
alternative characterizations. The following characterization of
graph parameters occurring in Theorem \ref{THM:LIM-CHAR} was proved
in \cite{LSz1}.

\begin{prop}\label{PROP:DAG-SEMI-1}
Let $f$ be a multiplicative, normalized graph parameter. Then the
following are equivalent:

\smallskip

{\rm(a)} $f$ is reflection positive;

\smallskip

{\rm(b)} $f$ is flatly reflection positive;

\smallskip

{\rm(c)} $f$ has nonnegative M\"obius transform;

\smallskip

{\rm(d)} $f=t(.,W)$, where $W$ is a graphon.

\smallskip

{\rm(e)} $f$ is the limit of homomorphism density functions.
\end{prop}

For graph parameters in Theorem \ref{THM:LIM-CHAR1}, we have the
following.

\begin{prop}\label{PROP:DAG-SEMI-2}
Let $f$ be an isolate-indifferent, normalized graph parameter. Then
the following are equivalent:

\smallskip

{\rm(a)} $f$ is reflection positive;

\smallskip

{\rm(b)} $f$ is flatly reflection positive;

\smallskip

{\rm(c)} $f$ has nonnegative M\"obius transform;

\smallskip

{\rm(d)} $f=\E(t(.,\mathbf{W}))$, where $\mathbf{W}$ is a random
graphon.

\smallskip

{\rm(e)} $f$ is in the convex hull of limits of homomorphism density
functions.
\end{prop}

While the proof here is similar, there are some differences, and we
include it for completeness.

\begin{proof}
(a)$\Rightarrow$(b) is trivial.

\smallskip

(b)$\Rightarrow$(c): The Lindstr\"om--Wilf Formula gives the
following diagonalization of $M_\flat(f,k)$: Let $Z$ denote the
$\FF'_k\times\FF'_k$ matrix defined by
$Z_{F_1,F_2}=\one_{F_1\subseteq F_2}$. Let $D$ be the diagonal matrix
with $D_{F,F}=f^\dag(F)$. Then $M_\flat(f,k)=Z\T D Z$. This implies
that $M_\flat(f,k)$ is positive semidefinite if and only if
$f^{\dag}\ge 0$ for all graphs with $k$ nodes.

\smallskip

(c)$\Rightarrow$(d): Let $f$ be an isolate-indifferent, normalized
graph parameter with nonnegative M\"obius transform. By Theorem
\ref{THM:LIM-CHAR1}, it defines a random graphon $\mathbf{W}$ such
that $f=\E(t(.,\mathbf{W}))$.

\smallskip

(d)$\Rightarrow$(e): By Theorem \ref{THM:LIM-CHAR1}, each
$t(.,\mathbf{W})$ is the limit of homomorphism density functions for
every $\mathbf{W}$.

\smallskip

(e)$\Rightarrow$(a): Every homomorphism density function $f$ is
reflection positive, and this is clearly inherited to their limits,
and then to the convex hull of these limits.
\end{proof}

The following propositions describe connections between
graph-theoretic and group-theoretic properties of countable random
graph models. They also indicate a connection with ergodic theory.

\begin{prop}\label{PROP:ERGOD}
A countable random graph model is consistent if and only if it
invariant.
\end{prop}

\begin{proof}
It is trivial that invariant countable random graph models are
consistent. Conversely, if a countable random graph model is
consistent, then it defines a consistent finite graph model, which in
turn defines a unique countable random graph model, independently of
the labeling of the nodes.
\end{proof}

\begin{prop}\label{PROP:ERGOD-2}
A consistent countable random graph model is local if and only if it
is ergodic.
\end{prop}

\begin{proof}
Let $\mu$ be an invariant probability measure on the Borel sets in
$\{0,1\}^{\binom{\N}{2}}$. By Proposition \ref{PROP:ERGOD} it is
consistent, and so by Theorem \ref{THM:LIM-CHAR1} it is defined by a
random graphon. If $\mu$ is ergodic, then $\mu$ is an extreme point
of all invariant distributions, and therefore this random graphon
must be concentrated on a single graphon. Thus Theorem
\ref{THM:LIM-CHAR} implies that $\mu$ is local.

Conversely, if $\mu$ is not ergodic, then $\mu=\frac12(\mu_1+\mu_2)$,
where $\mu_1,\mu_2$ are invariant probability measures and
$\mu_1\not=\mu_2$. Let $\Ge_1$ and $\Ge_2$ be random countable graphs
from the distributions $\mu_1$ and $\mu_2$, respectively, and let
$\Ge$ be $\Ge_1$ with probability $1/2$ and $\Ge_2$ with probability
$1/2$. Let $S\subseteq\N$ be a finite set and $F$ a labeled graph on
$|S|$ nodes such that $\Pr(\Ge_1[S]=F)\not=\Pr(\Ge_2[S]=F)$. Let
$T\subseteq\N$ be another set with $|T|=|S|$ and $T\cap S=\emptyset$.
Set $a_1=\Pr(\Ge_1[S]=F)=\Pr(\Ge_1[T]=F)$ (by invariance, these two
probabilities are equal), and define $a_2$ analogously.

Thus we have
\begin{align*}
\Pr(\Ge[S]&=F,\Ge[T]=F)-\Pr(\Ge[S]=F)\Pr(\Ge[T]=F)\\
&=\frac12\bigl(\Pr(\Ge_1[S]=F,\Ge_1[T]=F)+\Pr(\Ge_2[S]=F,\Ge_2[T]=F)\bigr)\\
&~~~- \frac14\bigl(\Pr(\Ge_1[S]=F)+\Pr(\Ge_2[S]=F)\bigr)
\bigl(\Pr(\Ge_1[T]=F)+\Pr(\Ge_2[T]=F)\bigr)\\
&=\frac12(a_1^2+a_2^2)-\frac14(a_1+a_2)^2=\frac14(a_1-a_2)^2>0.
\end{align*}
This shows that $\mu$ is not local.
\end{proof}

\section{Weak Positivstellensatz for graphs}

Let $x=\alpha_1 F_1 + \dots + \alpha_r F_r$ be any quantum graph. We
say that $x\ge 0$ if $t(x,W)=\sum_i\alpha_it(F_i,W)\ge 0$ for every
$W\in\WW_0$. Hence $x\ge 0$ if and only if $\sum_i\alpha_if(F_i)\ge
0$ for every multiplicative, reflection positive graph parameter $f$.
Proposition \ref{PROP:DAG-SEMI-2} implies that this is equivalent to
saying that $\sum_i\alpha_if(F_i)\ge 0$ for every
isolate-indifferent, reflection positive parameter $f$.

An easy example of quantum graphs $x\ge 0$ is any quantum graph of
the form $\sum_i y_i^2$, where the $y_i$ are $k$-labeled quantum
graphs for some $k\ge 0$ (and the labels are ignored after squaring).

One may ask whether every quantum graph $x\ge 0$ can be represented
this way. We don't know the answer, although based on the analogy of
polynomials, the answer is probably negative. However, we prove the
following weaker version, which is analogous to Lasserre's result
\cite{Las} asserting that positive polynomials are approximately sums
of squares.

\begin{theorem}\label{THM:W-POS}
Let $x$ be a quantum graph. Then $x\ge 0$ if and only if for every
$\eps>0$ there is a $k\ge 1$ and $y_1,\dots,y_m\in\GG_k$ such that
$\|x-y_1^2-\dots-y_m^2\|_1<\eps$.
\end{theorem}

\begin{proof}
For $n\ge k\ge 0$, let $\FF_k$ denote the set of $k$-labeled simple
graphs on $[k]$ (up to isomorphism). Let $\Fi_k$ denote the operator
mapping a matrix $\MM$ to its restriction to $\FF_k\times\FF_k$. Then
$\MM_k=\Fi_k\MM$ is the space of all symmetric $\FF_k\times\FF_k$
matrices, and $\PP_k=\Fi_k\PP$ is the positive semidefinite cone in
$\Fi_k\MM$. It is also clear that $\LL_k=\Fi_k\LL$ consists of those
matrices $A\in\MM_k$ in which $A_{F_1,G_1}=A_{F_2,G_2}$ whenever
$F_1G_1\simeq F_2G_2$. We set $\RR_k=\Fi_k\PP\cap\Fi_k\LL$. Clearly,
\begin{equation}\label{EQ:FIRR}
\Fi_k(\PP\cap\LL) \subseteq \RR_k,
\end{equation}
but equality may not hold in general.

We note that the entries of every matrix $A\in\RR_k$ are in
$[0,A_{\emptyset,\emptyset}]$. Indeed, looking at the $2\times 2$
submatrix formed by the rows corresponding to some $k$-labeled flat
graph $F$ and the $k$-labeled edgeless graph $U_k$. From
$A\in\Fi_k\LL$ it follows that $A_{U_k,F}=A){F,F}=A_{F,U_k}$, so
positive semidefiniteness implies that $A_{U_k,U_k}A_{F,F}\ge
A_{F,F}^2$. Since $A_{U_k,U_k}=A_{\emptyset,\emptyset}$ by
$A\in\Fi_k\LL$, we get that
$(A_{\emptyset,\emptyset}-A_{F,F})A_{F,F}\ge 0$, which implies that
$A_{F,F}\in[0,A_{\emptyset,\emptyset}]$.

For $k\le m$, we consider $\FF_k$ as a subset of $\FF_m$, by adding
$m-k$ isolated nodes labeled $k+1,\dots,m$. The corresponding
restriction operator on matrices we denote by $\Fi_{m,k}$.

We claim that the following weak converse of \eqref{EQ:FIRR} holds:
\begin{equation}\label{EQ:RQ}
\Fi_k(\PP\cap\LL) = \bigcap_{m\ge k} \Fi_{m,k}\RR_m.
\end{equation}
Indeed, let $A$ be a matrix that is contained in the right hand side.
Then for every $m\ge k$ we have a matrix $B_m\in\RR_m$ such that $A$
is a restriction of $B_m$. Now let $m\to\infty$; by selecting a
subsequence, we may assume that all entries of $B_m$ tend to a limit.
This limit defines a graph parameter $f$, which is normalized,
isolate-indifferent and flatly reflection positive. By Proposition
\ref{PROP:DAG-SEMI-2}, $f$ is reflection positive, and so the matrix
$M(f)$ is in $\PP\cap\LL$ and $\Fi_kM(f)=A$.

Let $x=\alpha_1 F_1 + \dots + \alpha_r F_r$. We may assume that
$|V(F_i)|=k$ for all $i$. Let $F_i'$ be obtained from $F_i$ by
labeling all its nodes. Let $A\in\MM_k$ denote the matrix
\[
A_{FG}=
  \begin{cases}
    \alpha_i, & \text{if $F=G=F_i$}, \\
    0, & \text{otherwise}.
  \end{cases}
\]
Then $x\ge 0$ means that $A\cdot Z\ge 0$ for all
$Z\in\Fi_k(\PP\cap\LL)$, in other words, $A$ is in the dual cone of
$\Fi_k(\PP\cap\LL)$. From \eqref{EQ:RQ} it follows that there are
diagonal matrices $A_m\in \MM_k$ such that $A_m\to A$ and $A_m\cdot
Y\ge 0$ for all $Y\in\Fi_{m,k}\RR_m$. In other words, $A_m\cdot
\Fi_{m,k}Z\ge 0$ for all $Z\in\RR_m$, which can also be written as
$\Fi_{m,k}^*A_m\cdot Z\ge 0$, where $\Fi_{m,k}^*:~\MM_k\to\MM_m$ is
the adjoint of the linear map $\Fi_{m,k}:~\MM_m\to\MM_k$. (This
adjoint acts by adding 0-s in all entries outside
$\FF_k\times\FF_k$.) So $\Fi_{m,k}^*A_m$ is in the polar cone of
$\RR_m=\PP_m\cap\LL_m$, which is $\PP_M^*+\LL_m^*$. The positive
semidefinite cone is self-polar. The linear space $\LL_m^*$ consists
of those matrices $B\in\MM_m$ for which
$\sum_{F_1,F_2}B_{F_1,F_2}=0$, where the summation extends over all
pairs $F_1,F_2\in\FF'_m$ for which $F_1F_2\simeq F_0$ for some fixed
graph $F_0$. Thus we have $\Fi_{m,k}A_m = P+L$, where $P$ is positive
semidefinite and $L\in\LL^*_m$. Since $P$ is positive semidefinite,
we can write it as $P=\sum_{k=1}^N v_kv_k\T$, where
$v_k\in\R^{\FF'_m}$. We can write this as
\[
\sum_{F_1,F_2\atop F_1F_2\simeq F_0}\sum_{k=0}^N v_{k,F_1}v_{k,F_2} =
  \begin{cases}
   (A_m)_{F_0,F_0},  & \text{if $F_1F_2\simeq F_0\in\FF_k$,} \\
    0, & \text{otherwise}.
  \end{cases}
\]
In other words,
\[
\sum_{k=1}^N \left(\sum_F v_{k,F}F\right)^2 = \sum_{F_0}
(A_m)_{F_0,F_0}F_0,
\]
which proves the Theorem.
\end{proof}

\end{document}